\documentclass[reqno]{amsart}
\usepackage{mathrsfs}
\usepackage{amssymb}
\usepackage{latexsym}
\usepackage{yfonts}
\usepackage{mathrsfs}
\usepackage{natbib}
\usepackage{amsmath}
\usepackage{graphicx,color}
\usepackage[pdftex,plainpages=false,colorlinks,hyperindex,bookmarksopen,linkcolor=red,citecolor=blue,urlcolor=blue]{hyperref}
\usepackage[final]{showkeys}
\usepackage{subfig}
\usepackage{bm}
\usepackage{wrapfig}
\usepackage{float}
\usepackage{lscape}
\DeclareMathAlphabet{\mathpzc}{OT1}{pzc}{m}{it}

\bibpunct{[}{]}{;}{n}{,}{,}

\newtheorem{te}{Theorem}[section]

\theoremstyle{definition}

\theoremstyle{os}
\newtheorem{os}[te]{Remark}

\theoremstyle{prop}
\newtheorem{prop}[te]{Proposition}

\theoremstyle{lem}

\theoremstyle{coro}

\numberwithin{equation}{section}

\allowdisplaybreaks

\def \l { \left( }
\def \r {\right) }
\def \ll { \left\lbrace }
\def \rr { \right\rbrace }

\begin{document}

	\title [Space-fractional higher-order heat-type equations] {Pseudoprocesses related to space-fractional higher-order heat-type equations}
	\author{Enzo Orsingher}
	\author{Bruno Toaldo} 
	\address{Department of Statistical Sciences, Sapienza University of Rome}
	\email{enzo.orsingher@uniroma1.it} \email{bruno.toaldo@uniroma1.it}
	\date{\today}

		\begin{abstract}
In this paper we construct pseudo random walks (symmetric and asymmetric) which converge in law to compositions of pseudoprocesses stopped at stable subordinators. We find the higher-order space-fractional heat-type equations whose fundamental solutions coincide with the law of the limiting pseudoprocesses. The fractional equations involve either Riesz operators or their Feller asymmetric counterparts. The main result of this paper is the derivation of pseudoprocesses whose law is governed by heat-type equations of real-valued order $\gamma>2$. The classical pseudoprocesses are very special cases of those investigated here.
		\end{abstract}
	
	\maketitle

\section{Introduction}

In this paper we consider pseudoprocesses related to different types of fractional higher-order heat-type equations. Our starting point is the set of higher-order equations of the form
\begin{equation}
\frac{\partial}{\partial t} u_m(x, t) \, = \, \kappa_m \frac{\partial^m}{\partial x^m} u_m(x, t), \qquad x \in \mathbb{R}, t>0, m\in \mathbb{N} > 2,
\label{11}
\end{equation}
whose solutions have been investigated by many outstanding mathematicians such as \citet{berna, levy, polia} and also, more recently, by means of the steepest descent method, by \citet{liwong}. In \eqref{11} the constant $\kappa_m$ is usually chosen in the form
\begin{equation}
\kappa_m \, = \, 
\begin{cases}
\pm 1, \qquad &m = 2n+1, \\
(-1)^{n+1}, &m=2n.
\end{cases}
\end{equation}
In our investigations we assume throughout that $\kappa_m = (-1)^n$ when $m = 2n+1$. Pseudoprocesses related to \eqref{11} have been constructed in the same way as for the Wiener process by \citet{dale1, dale2, krylov, lado, myamoto}. More recently pseudoprocesses related to \eqref{11} have been considered by \citet{debbi06, lachal2003, lachalpseudo, mazzucchi}. For equations of the form
\begin{align}
\frac{\partial}{\partial t} u_\gamma (x, t) \, = \, \frac{\partial^\gamma}{\partial |x|^\gamma} u_\gamma (x, t), \qquad x \in \mathbb{R}, t>0,
\label{13}
\end{align}
where $0 < \gamma \leq 2$, and $\frac{\partial^\gamma}{\partial |x|^\gamma}$ is the Riesz operator, the fundamental solution has the form of the density of a symmetric stable process as Riesz himself has shown. For $\gamma > 2$ the equation \eqref{13} was studied by Debbi (see \cite{debbi06, debbi}) who proved the sign-varying character of the corresponding solutions.

For asymmetric fractional operators of the form
\begin{equation}
^FD^{\gamma, \theta} \, = \, -\left[ \frac{\sin \frac{\pi}{2}(\gamma - \theta)}{\sin \pi \gamma} \frac{^+\partial^\gamma}{\partial x^\gamma} + \frac{\sin \frac{\pi}{2}(\gamma + \theta)}{\sin \pi \gamma} \frac{^-\partial^\gamma}{\partial x^\gamma} \right]
\label{14}
\end{equation}
the equation
\begin{equation}
\frac{\partial}{\partial t} u_{\gamma, \theta} (x, t) \, = \, ^FD^{\gamma, \theta} u_{\gamma, \theta}(x, t), \qquad x \in \mathbb{R}, t>0, 0 < \gamma \leq 2,
\label{15}
\end{equation}
was studied by \citet{Feller52} who proved that the fundamental solution to \eqref{15} is the law of an asymmetric stable process of order $\gamma$. The fractional derivatives appearing in \eqref{14} are the Weyl fractional derivatives defined as
\begin{align}
&\frac{^+\partial^\gamma}{\partial x^\gamma} u(x) \, = \, \frac{1}{\Gamma (m-\gamma)} \frac{d^m}{dx^m} \int_{-\infty}^x \frac{u(y)}{(x-y)^{\gamma +1 -m}} dy \notag  \\
& \frac{^-\partial^\gamma}{\partial x^\gamma} u(x) \, = \, \frac{1}{\Gamma (m-\gamma)} \frac{d^m}{dx^m} \int_{x}^\infty \frac{u(y)}{(y-x)^{\gamma + 1 -m}} dy
\label{16}
\end{align}
where $m-1 < \gamma < m$.
The Riesz fractional derivatives appearing in \eqref{13} are combinations of the Weyl's derivatives \eqref{16} and are defined as
\begin{align}
\frac{\partial^\gamma}{\partial |x |^\gamma} \, = \, -\frac{1}{2\cos \frac{\pi \gamma}{2}} \left[ \frac{^+\partial^\gamma}{\partial x^\gamma} + \frac{^-\partial^\gamma}{\partial x^\gamma} \right].
\end{align}

This paper is devoted to pseudoprocesses related to fractional equations of the form \eqref{13} and \eqref{15} when $\gamma > 2$. Of course, this implies that Weyl's fractional derivatives \eqref{16} are considered in the case $\gamma > 2$. The fundamental solutions of these equations are sign-varying as in the case of higher-order heat-type equations \eqref{11} studied in the literature (compare with \cite{debbi06}).

Fractional equations arise, for example, in the study of thermal diffusion in fractal and porous media (\citet{nigma, saiche}). Other fields of application of fractional equations can be found in \citet{debbi06}. Higher-order equations emerge in many contexts as in trimolecular chemical reactions (\citet{gardiner} page 295) and in the linear approximation of the Korteweg De Vries equation (see \citet{beghin4}).

In our paper we study pseudo random walks (for the definitions and properties of pseudo random walks and variables see \citet{lachalpseudo}) of the form
\begin{equation}
W^{\gamma, 2k\beta} (t) \, = \, \sum_{j=1}^{N \l t\gamma^{-2k\beta} \r} U_j^{2k}(1) Q_j^{\gamma, 2k\beta}
\label{19}
\end{equation}
where the r.v.'s $Q_j^{\gamma, 2k\beta}$ are independent from the Poisson process $N$, from the pseudo r.v.'s $U_j^{2k}(1)$ and from each other and have distribution for $0 < \beta < 1$, $\gamma > 0$, $k \in \mathbb{N}$,
\begin{align}
\Pr \ll Q_j^{\gamma, 2k\beta} > w \rr \, = \, \begin{cases} 1, \qquad & w < \gamma, \\ \l \frac{\gamma}{w} \r^{2k\beta}, & w \geq \gamma. \end{cases}
\label{110}
\end{align}
The $U_j^{2k}(1)$ are independent pseudo r.v.'s with law $u_{2k} (x, 1)$ with Fourier transform
\begin{align}
\int_{-\infty}^\infty e^{i\xi x} u_{2k}(x, 1) \, dx \, = \, e^{-|\xi|^{2k}}.
\end{align}
The Poisson process $N$ appearing in \eqref{19} is homogeneous and has rate $\lambda = \frac{1}{\Gamma (1-\beta)}$. We prove that
\begin{align}
\lim_{\gamma \to 0} W^{\gamma, 2k\beta} (t) \, \stackrel{\textrm{law}}{=} \, U^{2k} \l H^\beta (t) \r
\label{112}
\end{align}
where $U^{2k}$ is the pseudoprocess of order $2k$ related to the heat-type equation \eqref{11} for $m=2k$ and $H^\beta$ is a stable subordinator of order $\beta \in (0,1)$ independent from $U^{2k}$. We show that the law of \eqref{112} is the fundamental solution to
\begin{equation}
\frac{\partial}{\partial t} v_{2k\beta} (x, t) \, = \, \frac{\partial^{2k\beta}}{\partial |x|^{2k\beta}} v_{2k\beta} (x, t), \qquad x \in \mathbb{R}, t>0, \beta \in (0,1), k \in \mathbb{N}.
\end{equation}
In other words, we are able to construct pseudoprocesses of order $\gamma >2$ in the form of integer-valued pseudoprocesses stopped at stable distributed times as the limit of suitable pseudo random walks. We consider also pseudo random walks of the form
\begin{equation}
\sum_{j=0}^{N \l t \gamma^{-\beta(2k+1)} \r} \epsilon_j U_j^{2k+1} (1) Q_j^{\gamma, \beta(2k+1)}
\label{114}
\end{equation}
where the $Q_j^{\gamma, \beta(2k+1)}$ have distribution \eqref{110} (suitably adjusted), $U_j^{2k+1}$ (1) is an odd-order pseudo random variable with law $u_{2k+1} (x, 1)$ and Fourier transform
\begin{align}
\int_{-\infty}^\infty e^{i\xi x} u_{2k+1} (x, 1) \, dx \, = \, e^{-i\xi^{2k+1}}
\end{align}
and the $\epsilon_j$'s are random variables which take values $\pm 1$ with probability $p$ and $q$. All the variables in \eqref{114} are independent from each other and also independent from the Poisson process $N$ with rate $\lambda = \frac{1}{\Gamma (1-\beta)}$. In this case we are able to show that
\begin{equation}
\lim_{\gamma \to 0} W^{\gamma, (2k+1)\beta} (t) \, \stackrel{\textrm{law}}{=} \, U_1^{2k+1} \l H_1^\beta (pt) \r - U_2^{2k+1} \l H_2^\beta (qt) \r
\label{116}
\end{equation}
where $H_j^\beta$, $j=1,2$, are independent stable subordinators  independent also from the pseudoprocesses $U_1$, $U_2$. We prove that the law of \eqref{116} satisfies the higher-order fractional equation
\begin{equation}
\frac{\partial}{\partial t} w_{\beta (2k+1)} (x, t) \, = \, \mathfrak{R} w_{\beta (2k+1)} (x, t), \qquad x \in \mathbb{R}, t>0,
\label{117v}
\end{equation}
where
\begin{equation}
\mathfrak{R} \, = \, -\frac{1}{\cos \frac{\beta \pi}{2}} \left[ p e^{i\pi\beta k} \frac{^+\partial^{\beta(2k+1)}}{\partial x^{\beta(2k+1)}} +qe^{-i\pi\beta k} \frac{^-\partial^{\beta(2k+1)}}{\partial x^{\beta (2k+1)}} \right].
\label{117}
\end{equation}
The Fourier transform of the  fundamental solution of \eqref{117v} reads
\begin{equation}
\widehat{w}_{\beta(2k+1)} (\xi, t) \, = \, e^{-t|\xi|^{\beta(2k+1)} \l 1-i \textrm{ sign}(\xi) \, (p-q) \tan \frac{\beta \pi}{2} \r}
\label{119}
\end{equation}
We note that \eqref{119} corresponds to the Fourier transform of the law of \eqref{116} with a suitable change of the time-scale that is
\begin{align}
&\mathbb{E}\exp \ll i\xi  \left[ U_1^{2k+1} \l H_1^\beta \l \frac{pt}{\cos \frac{\beta \pi}{2}} \r \r - U_2^{2k+1} \l H_2^\beta \l \frac{qt}{\cos \frac{\beta \pi}{2}} \r \r \right] \rr \notag \\
 = \, & e^{-t|\xi |^{\beta(2k+1)} (1-i \textrm{ sign}(\xi) \, (p-q) \, \tan \frac{\beta \pi}{2}}
\end{align}
The mean value here and below must be understood with respect to the signed measure of the pseudoprocess (see for example \cite{debbi06}).
We study also the pseudoprocesses governed by the equation
\begin{align}
\frac{\partial}{\partial t} z_{\beta(2k+1), \theta} (x, t) \, = \, ^FD^{\beta(2k+1), \theta} z_{\beta(2k+1), \theta} (x, t)
\end{align}
where $^FD^{\beta(2k+1), \theta}$ is the operator defined in \eqref{14} with $\gamma$ replaced by $\beta (2k+1)$.
Also in this case we study continuous-time random walks whose limit  has Fourier transform equal to
\begin{equation}
\mathbb{E}e^{i\xi Z^{\beta(2k+1), \theta}} \, = \, e^{-t|\xi |^{\beta(2k+1)}e^{\frac{i\pi\theta}{2} \textrm{ sign}(\xi)}}, \qquad \beta \in (0,1), k \geq 1, -\beta < \theta < \beta.
\end{equation}
When we take into account pseudo random walks constructed by means of even-order pseudo random variables we arrive at limits $Z^{2\beta k, \theta}(t)$, $t>0$, with Fourier transform
\begin{equation}
\mathbb{E}e^{i\xi Z^{2\beta k, \theta} (t)} \, = \, e^{-t|\xi |^{2k\beta} \frac{\cos \frac{\pi}{2}\theta}{\cos \frac{\pi}{2}\beta}}
\end{equation}
which shows the symmetric structure of the limiting pseudoprocess.

\subsection{List of symbols}
For the reader's convenience we give a short list of the most important symbols and definitions appearing in the paper.
\begin{enumerate}
\item[$\bullet$] The right Weyl fractional derivative for $m -1<\gamma < m$, $m \in \mathbb{N}$, $x \in \mathbb{R}$
\begin{equation}
\frac{^+\partial^\gamma}{\partial x^\gamma} u(x, t)  \, = \, \frac{1}{\Gamma (m-\gamma)} \frac{d^m}{d x^m} \int_{-\infty}^x \frac{u(y, t)}{(x-y)^{\gamma + 1 - m}} dy
\end{equation}
\item[$\bullet$] The left Weyl fractional derivative for $m -1<\gamma < m$, $m \in \mathbb{N}$, $x \in \mathbb{R}$,
\begin{equation}
\frac{^-\partial^\gamma}{\partial x^\gamma} u(x, t)  \, = \, \frac{(-1)^m}{\Gamma (m-\gamma)} \frac{d^m}{d x^m} \int_x^\infty \frac{u(y, t)}{(y-x)^{\gamma + 1 - m}} dy
\end{equation}
\item[$\bullet$] The Riesz fractional derivative for $m-1 < \gamma < m$, $m \in \mathbb{N}$, $x \in \mathbb{R}$,
\begin{equation}
\frac{\partial^\gamma}{\partial |x|^\gamma} \, = \, - \frac{1}{2\cos \frac{\gamma \pi}{2}} \left[ \frac{^+\partial^\gamma}{\partial x^\gamma} + \frac{^-\partial^\gamma}{\partial x^\gamma} \right]
\end{equation}
\item[$\bullet$] We introduce the operator $\mathfrak{R}$, for $\beta \in (0,1)$, $k \in \mathbb{N}$, $p,q \in [0,1] : p + q = 1$, $x \in \mathbb{R}$, 
\begin{equation}
\mathfrak{R} \, = \, -\frac{1}{\cos \frac{\beta \pi }{2}} \left[ p e^{i\pi\beta k} \frac{^+\partial^{\beta(2k+1)}}{\partial x^{\beta(2k+1)}} + q e^{-i\pi\beta k} \frac{^-\partial^{\beta(2k+1)}}{\partial x^{\beta(2k+1)}}  \right]
\end{equation}
\item[$\bullet$] The Feller derivative for $m-1 < \gamma < m$, $m \in \mathbb{N}$, $\theta > 0$, $x \in \mathbb{R}$,
\begin{equation}
^FD^{\gamma, \theta} \, = \, - \left[ \frac{\sin \frac{\pi}{2}(\gamma - \theta)}{\sin \pi\gamma} \frac{^+\partial^\gamma}{\partial x^\gamma} + \frac{\sin \frac{\pi}{2}(\gamma + \theta)}{\sin \pi \gamma} \frac{^-\partial^\gamma}{\partial x^\gamma} \right]
\end{equation}
\item[$\bullet$] $U^{m}(t)$, $t>0$ is a pseudoprocess of order $m \in \mathbb{N}$ with law $u_m (x, t)$, $x \in \mathbb{R}$, $t>0$, governed by \eqref{11}
\item[$\bullet$] $H^\beta (t)$ is a stable subordinator of order $\beta \in (0,1)$ with probability density $h_\beta (x, t)$, $x \geq0$, $t\geq 0$.

\end{enumerate}

\section{Preliminaries and auxiliary results}
In this paper we consider higher-order heat-type equations where the space derivative is fractional in different ways.
\subsection{Weyl fractional derivatives.} First of all we consider equations of the form
\begin{equation}
\frac{\partial}{\partial t} u_\gamma (x, t) \, = \, \frac{^{\pm}\partial^\gamma}{\partial x^\gamma} u_\gamma (x, t), \qquad x \in \mathbb{R}, t>0, \gamma >0,
\end{equation}
where $\frac{^{\pm}\partial^\gamma}{\partial x^\gamma}$ are the space-fractional Weyl derivatives defined as
\begin{align}
\frac{^+\partial^\gamma}{\partial x^\gamma} u_\gamma (x, t) \,  = \, \frac{1}{\Gamma (m-\gamma)} \frac{d^m}{dx^m} \int_{-\infty}^x \frac{u(z, t) \, dz}{(x-z)^{\gamma - m + 1}},  \quad  m-1 < \gamma < 1, m \in \mathbb{N}, 
\label{weyldestra}
\end{align}
\begin{align}
\frac{^-\partial^\gamma}{\partial x^\gamma} u_\gamma (x, t) \, = \, \frac{(-1)^m}{\Gamma (m-\gamma)} \frac{d^m}{dx^m} \int_x^\infty \frac{u(z, t) \, dz}{(z-x)^{\gamma - m +1}}, \quad m - 1 < \gamma < m, m \in \mathbb{N}.
\label{weylsinistra}
\end{align}
In our analysis the following result on the Fourier transforms of Weyl derivatives is very important.
\begin{te}[\cite{samko}, page 137]
\label{teoremasimboloweyl}
The Fourier transforms of \eqref{weyldestra} and \eqref{weylsinistra} read
\begin{align}
\int_{-\infty}^\infty dx \, e^{i\xi x} \frac{^+\partial^\gamma}{\partial x^\gamma} u(x, t) \, = \, (-i\xi)^\gamma \, \widehat{u} (\xi, t) \, = \, |\xi |^{\gamma} e^{-\frac{i\pi \gamma}{2} \textrm{ sign} (\xi)} \, \widehat{u} (\xi, t),
\label{simboloweyldestra}
\end{align}
\begin{equation}
\int_{-\infty}^\infty dx \, e^{i\xi x} \frac{^-\partial^\gamma}{\partial x^\gamma} u(x, t) \, = \, (i\xi)^\gamma \, \widehat{u} (\xi, t) \, = \, |\xi |^{\gamma} e^{\frac{i\pi \gamma}{2} \textrm{ sign} (\xi)} \, \widehat{u} (\xi, t).
\label{simboloweylsinistra}
\end{equation}
Clearly $\widehat{u}(\xi, t)$ is the $x$-Fourier transform of $u(x, t)$.
\end{te}
\begin{proof}
We give a sketch of the proof of \eqref{simboloweyldestra} with some details.
\begin{align}
\int_{-\infty}^\infty dx \, e^{i\xi x} \frac{^+\partial^\gamma}{\partial x^\gamma} u(x, t) \, = \, & \int_{-\infty}^\infty dx \, e^{i\xi x} \left[ \frac{1}{\Gamma (m-\gamma)} \frac{\partial^m}{\partial x^m} \int_{-\infty}^x dz \frac{u(z, t)}{(x-z)^{\gamma - m + 1}} \right] \notag \\
= \, & \int_{-\infty}^\infty dx \, e^{i\xi x} \left[ \frac{1}{\Gamma (m-\gamma)} \int_0^\infty dz \frac{\partial^m}{\partial x^m} \frac{u(x-z,t)}{z^{\gamma - m + 1}}  \right] \notag \\
= \, & \int_{-\infty}^\infty dw \, e^{i\xi w} \frac{\partial^m}{\partial w^m} u(w, t) \; \frac{1}{\Gamma (m-\gamma)} \int_0^\infty dz \, e^{i\xi z} z^{m-\gamma - 1} \notag \\
= \, & (-i\xi)^m \int_{-\infty}^\infty e^{i\xi w} u(w, t) \, dw \; \frac{1}{\Gamma (m-\gamma)} \int_0^\infty dz \, e^{i\xi z} z^{m-\gamma - 1}.
\label{oddio}
\end{align}
The result
\begin{equation}
 \frac{(-i\xi)^m}{\Gamma (m-\gamma)} \int_0^\infty dz \, e^{i\xi z} \, z^{m-\gamma -1} \, = \, |\xi |^\gamma e^{-\frac{i\pi }{2} \textrm{ sign}(\xi)}
\end{equation}
can be obtained for example by applying the Cauchy integral Theorem (see \citet{samko} page 138).
\end{proof}

\subsection{Riesz fractional derivatives} By means of the Weyl fractional derivatives we arrive at the Riesz fractional derivative, for $m-1 < \gamma < m$, $m \in \mathbb{N}$,
\begin{align}
\frac{\partial^{\gamma}}{\partial |x|^\gamma} u(x, t) =  & -\frac{\frac{\partial^m}{\partial x^m}}{2\cos \frac{\pi \gamma}{2} \Gamma (m-\gamma)}  \left[ \int_{-\infty}^x \frac{u(y, t) \, dy}{(x-y)^{\gamma - m + 1}} +  \int_x^\infty \frac{(-1)^m \, u(y, t) \, dy}{(y-x)^{\gamma - m + 1}} \right] \notag \\
= \, & - \frac{1}{2\cos \frac{\pi \gamma}{2} } \l \frac{^+\partial^\gamma}{\partial x^\gamma} + \frac{^-\partial^\gamma}{\partial x^\gamma} \r \, u(x, t)
\end{align}
In view of Theorem \ref{teoremasimboloweyl} we have that, for $\gamma > 0$, $\gamma \notin \mathbb{N}$,
\begin{align}
\int_{-\infty}^\infty dx \, e^{i\xi x} \frac{\partial^\gamma}{\partial |x|^\gamma} u(x, t) \, = \, & - \frac{1}{2\cos \frac{\pi \gamma}{2}} \left[ |\xi |^\gamma e^{-\frac{i\pi \gamma}{2} \textrm{ sign} (\xi)} + |\xi |^\gamma e^{\frac{i\pi \gamma}{2} \textrm{ sign}(\xi)}  \right] \widehat{u}(\xi, t) \notag \\
= \, & -|\xi |^{\gamma} \, \widehat{u} (\xi, t).
\end{align}
\begin{os}
The general fractional higher-order heat equation
\begin{equation}
\frac{\partial}{\partial t} u_\gamma (x, t) \, = \, \frac{\partial^\gamma}{\partial |x|^\gamma} u_\gamma (x, t), \qquad x \in \mathbb{R}, t>0,
\end{equation}
has solution whose Fourier transform reads
\begin{equation}
\widehat{u_\gamma} (\xi, t) \, = \, e^{-t |\xi |^\gamma}.
\label{due11}
\end{equation}
For $0 < \gamma < 2$, \eqref{due11} corresponds to the characteristic function of the symmetric stable processes (this is a classical result due to M. Riesz himself).
\end{os}

\section{From pseudo random walks to fractional pseudoprocesses}
We consider in this section continuous-time pseudo random walks with steps which are pseudo random variables, that is measurable functions endowed with signed measures, and with total mass equal to one (see \cite{lachalpseudo}). In order to obtain in the limit pseudoprocesses  whose signed law satisfies higher-order fractional equations we must construct sums of the form
\begin{equation}
\sum_{j=0}^{N \l t\gamma^{-\beta (2k+1)} \r} \epsilon_j \, U_j^{2k+1} (1) \, Q^{\gamma, \beta (2k+1)}_j, \qquad \beta \in (0,1), k \in \mathbb{N}, \gamma > 0,
\label{passeggiata}
\end{equation}
where
\begin{equation}
\epsilon_j \, = \, 
\begin{cases}
1, \quad & \textrm{with probability } p, \\
-1,  &  \textrm{with probability } q,
\end{cases}
\qquad p + q = 1.
\end{equation}
The r.v.'s $Q_j^{\gamma, \beta(2k+1)}$ have probability distributions, for $\beta \in (0,1)$, $k \in \mathbb{N}$,
\begin{equation}
\Pr \ll Q_j^{\gamma, \beta(2k+1)} > w \rr \, = \, \begin{cases} \l \frac{\gamma}{w} \r^{\beta (2k+1)}, \qquad & \textrm{for } w > \gamma \\
1, & \textrm{for } w < \gamma.
\end{cases}
\end{equation}
The Poisson process $N(t)$, $t>0$, appearing in \eqref{passeggiata} is homogeneous with rate
\begin{equation}
\lambda = \frac{1}{\Gamma (1-\beta)}, \qquad \beta \in (0,1).
\end{equation}
The pseudo random variables (see \citet{lachalpseudo}) $U^{2k+1}_j (1)$ have law with Fourier transform
\begin{equation}
\int_{-\infty}^\infty dx \, e^{i\xi x} u_{2k+1} (x, 1) \, = \, e^{-i\xi^{2k+1}}
\end{equation}
and the function $u_{2k+1} (x, t)$, $x \in \mathbb{R}$, $t>0$, is the fundamental solution to the odd-order heat-type equation, for $k \in \mathbb{N}$,
\begin{align}
\begin{cases}
\frac{\partial}{\partial t} u_{2k+1} (x, t) \, = \, (-1)^k \frac{\partial^{2k+1}}{\partial x^{2k+1}} u_{2k+1} (x, t), \qquad x \in \mathbb{R}, t>0, \\
u_{2k+1} (x, 0) \, = \, \delta(x).
\end{cases}
\label{oddorderpseudoprocessesgovern}
\end{align}
There is a vast literature devoted to odd-order heat-type equations of the form \eqref{oddorderpseudoprocessesgovern}, to the behaviour of their solutions, and to the related pseudoprocesses (\citet{beghin4, lachal2003, enzolit, ecporsdov}). 

The r.v.'s and pseudo r.v.'s appearing in \eqref{passeggiata} are independent and also independent from each other. We say that two pseudo r.v.'s (or pseudoprocesses) with signed density $u_m^1$, $u_m^2$, are independent if the Fourier transform $\mathcal{F}$ of the convolution $u_m^1 * u_m^2 $ factorizes, that is
\begin{align}
\mathcal{F} \left[ u_m^1 \, * \, u_m^2 \right] (\xi) \, = \, \mathcal{F} \left[ u_m^1 \right] (\xi) \, \mathcal{F} \left[ u_m^2 \right] (\xi).
\end{align}
We are now able to state the first theorem of this section.
\begin{te}
\label{pseudowalkdispari}
The following limit in distribution holds true
\begin{align}
\lim_{\gamma \to 0} \sum_{j=0}^{N \l t\gamma^{-\beta(2k+1)} \r} \epsilon_j \, U_j^{2j+1} (1) \, Q_j^{\gamma, \beta(2k+1)} \, \stackrel{\textrm{ law }}{=} \, U_1^{2k+1} \l H_1^\beta (pt) \r - U_2^{2k+1} \l H_2^\beta (qt) \r,
\label{limitepasseggiatadispari}
\end{align}
where $H_1^\beta$ and $H_2^\beta$ are independent positively-skewed stable processes of order $0 < \beta < 1$ while $U_1^{2k+1}$ and $U_2^{2k+1}$ are independent pseudoprocesses of order $2k+1$. All the random variables $N(t)$, $t>0$, $\epsilon_j$, $Q_j^{\gamma, \beta(2k+1)}$ are independent and also independent from the pseudo random variables $U_j^{2k+1} (1)$. The Fourier transform of the limiting pseudoprocess reads
\begin{equation}
\mathbb{E}e^{i\xi U_1^{2k+1} \l H_1^\beta (pt) \r - U_2^{2k+1} \l H_2^\beta (qt) \r } \, = \, e^{-t |\xi |^{\beta(2k+1)} \l \cos \frac{\beta \pi}{2} - i \textrm{ sign}(\xi) \, (p-q) \, \sin \frac{\beta \pi}{2} \r}.
\end{equation}
\end{te}
\begin{proof}
In view of the independence of the r.v's and pseudo random variables appearing in \eqref{limitepasseggiatadispari} we have that
\begin{align}
& \mathbb{E} e^{i\xi\sum_{j=0}^{N \l t\gamma^{-\beta (2k+1)} \r} \epsilon_j U^{2k+1}_j (1) \, Q_j^{\gamma, \beta(2k+1)}} \notag \\
= \, & \mathbb{E} \left[ \mathbb{E} \l e^{i\xi \epsilon U^{2k+1} (1) Q^{\gamma, \beta (2k+1)}} \r^{N \l t \gamma^{-\beta (2k+1)} \r} \right] \notag \\
= \, & \exp \ll - \frac{\lambda t}{\gamma^{\beta (2k+1)}} \l 1-\mathbb{E}e^{i\xi \epsilon U^{2k+1}(1) Q^{\gamma, \beta (2k+1)}} \r \rr \notag \\
= \, & \exp \ll - \frac{\lambda t}{\gamma^{\beta (2k+1) } }  \l 1-p \mathbb{E}e^{i\xi U^{2k+1 } (1) Q^{\gamma, \beta (2k+1)}} - q \mathbb{E}e^{-i\xi U^{2k+1} (1) Q^{\gamma, \beta (2k+1)}} \r \rr \notag \\
= \, & \exp \ll - \frac{\lambda t}{\gamma^{\beta (2k+1) } } \l p+q -p \mathbb{E}e^{i\xi U^{2k+1 } (1) Q^{\gamma, \beta (2k+1)}} - q \mathbb{E}e^{-i\xi U^{2k+1} (1) Q^{\gamma, \beta (2k+1)}}  \r  \rr \notag \\
= \, & \exp \ll - \frac{\lambda t}{\gamma^{\beta (2k+1) } } \l p \l 1-\mathbb{E}e^{i\xi U^{2k+1} (1) Q^{\gamma, \beta (2k+1)}} \r + q \l 1-\mathbb{E}e^{-i\xi U^{2k+1} (1) Q^{\gamma, \beta (2k+1)}} \r \r \rr.
\end{align}
We observe that
\begin{align}
& p \l 1-\mathbb{E}e^{i\xi U^{2k+1} (1) Q^{\gamma, \beta (2k+1)}} \r + q \l 1-\mathbb{E}e^{-i\xi U^{2k+1} (1) Q^{\gamma, \beta (2k+1)}} \r \notag \\
= \, & p \int_{\gamma}^{\infty} dw \l 1-  \frac{\gamma^{\beta (2k+1)} \beta (2k+1)}{w^{\beta (2k+1) + 1}}  e^{i\xi^{2k+1} w^{2k+1}} \r \notag \\
& + q \int_\gamma^{\infty} dw \l 1-  \frac{\gamma^{\beta (2k+1)} \beta (2k+1)}{w^{\beta (2k+1) + 1}} e^{-i\xi^{2k+1} w^{2k+1}} \r
\end{align}
and therefore
\begin{align}
& \mathbb{E} e^{i\xi\sum_{j=0}^{N \l t\gamma^{-\beta (2k+1)} \r} \epsilon_j U^{2k+1}_j (1) \, Q_j^{\gamma, \beta(2k+1)}} \, = \notag \\
= \, & \exp \ll  -\frac{\lambda t}{\gamma^{\beta (2k+1)}} \left[ p \int_{\gamma}^{\infty} dw \l 1-  \frac{\gamma^{\beta (2k+1)} \beta (2k+1)}{w^{\beta (2k+1) + 1}} e^{i\xi^{2k+1} w^{2k+1}} \r \right. \right. \notag \\
& \left. \left. + q \int_\gamma^{\infty} dw \l 1-  \frac{\gamma^{\beta (2k+1)} \beta (2k+1)}{w^{\beta (2k+1) + 1}} e^{-i\xi^{2k+1} w^{2k+1}} \r \right] \rr \notag \\
= \, & \exp \ll \frac{- \lambda t}{\gamma^{\beta (2k+1)}} \left[ p \l 1-e^{i(\xi \gamma)^{2k+1}} \r  -p \,  i \xi^{2k+1}  (2k+1) \int_\gamma^\infty \frac{ dw \, \gamma^{\beta(2k+1)}  e^{i(\xi w)^{2k+1}}}{ w^{\beta(2k+1)-2k} }   \right. \right. \notag \\  
& \left. \left. +q \l 1-e^{-i(\xi \gamma)^{2k+1}} \r  + q \, i \xi^{2k+1}  (2k+1) \int_\gamma^\infty \frac{dw \, \gamma^{\beta (2k+1)} \, e^{-i(\xi w)^{2k+1}}}{w^{\beta (2k+1) - 2k}} \right] \rr.
\end{align}
By taking the limit we get that
\begin{align}
& \lim_{\gamma \to 0} \mathbb{E} e^{i\xi\sum_{j=0}^{N \l t\gamma^{-\beta (2k+1)} \r} \epsilon_j U^{2k+1}_j (1) \, Q_j^{\gamma, \beta(2k+1)}} \, = \notag \\
 = \, &\exp \left[ -\lambda t (2k+1) \l -p  i\xi^{2k+1}  \int_0^\infty \frac{dw \, e^{i(\xi w)^{2k+1} }}{w^{\beta (2k+1) - 2k}} + q  i \xi^{2k+1}   \int_0^\infty \frac{dw \, e^{-i(\xi w)^{2k+1} }}{w^{\beta (2k+1) -2k}}   \r \right] \notag \\
 = \, & e^{-\lambda t \Gamma (1-\beta) \left[  p \l -i\xi^{2k+1} \r^\beta   + q (i\xi^{2k+1})^\beta   \right]}.
\end{align}
By setting $\lambda = \frac{1}{\Gamma (1-\beta)}$ we obtain
\begin{align}
\lim_{\gamma \to 0} \mathbb{E} e^{i\xi\sum_{j=0}^{N \l t\gamma^{-\beta (2k+1)} \r} \epsilon_j U^{2k+1}_j (1) \, Q_j^{\gamma, \beta(2k+1)}} \, = \, & e^{-t \l p |\xi |^{\beta (2k+1)} e^{-\frac{i\pi \beta}{2}\textrm{ sign}(\xi)} + q |\xi |^{\beta(2k+1)} e^{\frac{i\pi \beta}{2}\textrm{ sign}(\xi)}  \r} \notag \\
= \, & e^{-t| \xi |^{\beta(2k+1)} \l \cos \frac{\pi \beta}{2} -i \textrm{ sign}(\xi) \, (p-q) \sin \frac{\pi \beta}{2} \r}
\label{finalmente}
\end{align}
Now we consider the Fourier transform of the law of the pseudoprocess 
\begin{equation}
V^{(2k+1)\beta} (t) \, = \,  U_1^{2k+1} \l H_1^\beta (pt) \r -  \, U_2^{2k+1} \l H_2^\beta (qt) \r
\end{equation}
which reads
\begin{align}
\mathbb{E} e^{i\xi V^{(2k+1)\beta} (t)} \, = \, & \mathbb{E} e^{i\xi  U_1^{2k+1} \l H_1^\beta (pt) \r } \mathbb{E} e^{-i\xi  U_2^{2k+1} \l H_2^\beta (qt) \r } \notag \\
= \, & \l \int_0^\infty ds \, e^{i\xi^{2k+1} s} \, h_\beta^1 (s, pt) \r \l \int_0^\infty dz \, e^{-i \xi^{2k+1} z} \, h_\beta^2 (z, qt) \r \notag \\
= \, & e^{-t \, p \l -  i \xi^{2k+1} \r^\beta} e^{-t \, q \l  i\xi^{2k+1}\r^\beta} \notag \\
= \, & e^{-t \l p |\xi |^{\beta(2k+1)} e^{-\frac{i\beta \pi}{2}\textrm{ sign}(\xi)} + q |\xi |^{\beta (2k+1)} e^{\frac{i\beta \pi}{2}\textrm{ sign}(\xi)} \r} \notag \\
= \, & e^{-t |\xi |^{\beta(2k+1)} \l \cos \frac{\pi \beta}{2} -i \textrm{ sign}(\xi) \, (p-q) \sin \frac{\pi \beta}{2} \r},
\label{caratteristicapseudosimm}
\end{align}
and coincides with \eqref{finalmente}.
\end{proof}
\begin{os}
If $\beta = \frac{1}{2k+1}$ the Fourier transform \eqref{caratteristicapseudosimm} becomes
\begin{equation}
\mathbb{E} e^{i\xi  U_1^{2k+1} \l H_1^\beta (pt) \r } \mathbb{E} e^{-i\xi  U_2^{2k+1} \l H_2^\beta (qt) \r } \, = \, e^{-t|\xi| \cos \frac{\pi}{2(2k+1)} + i t \xi \sin \frac{\pi}{2(2k+1)}}
\end{equation}
which corresponds to the characteristic function of a Cauchy r.v. with position parameter equal to $t (p-q) \sin \frac{\pi}{2(2k+1)}$ and scale parameter $t \cos \frac{\beta}{2(2k+1)}$. This slightly generalizes result 1.4 of \citet{ecporsdov}.
\end{os}

For even-order pseudoprocesses we have the following limit in distribution.
\begin{te}
\label{pseudowalkpari}
If $U^{2k}(t)$, $t>0$, is an even-order pseudoprocess and $N(t)$, $t>0$, is a homogeneous Poisson process, independent from $U^{2k}(t)$, $t>0$, we have that
\begin{align}
\lim_{\gamma \to 0} \sum_{j=0}^{N\l t \gamma^{-2k\beta} \r} U_j^{2k}(1)Q_j^{\gamma, 2k\beta} \, \stackrel{\textrm{ law }}{=} \, U^{2k} \l H^\beta (t) \r, \qquad t>0,
\end{align}
where $H^\beta$ is a stable subordinator of order $\beta \in (0,1)$ and $Q_j^{\gamma, 2k\beta}$ are i.i.d. random variables with distribution
\begin{align}
\Pr \ll Q_j^{\gamma, 2k\beta} > w \rr \, = \, \begin{cases}
1, \qquad & w < \gamma, \\
\l \frac{\gamma}{w} \r^{2k\beta}, & w > \gamma.
\end{cases}
\end{align}
The pseudoprocess $U^{2k}(t)$ is governed by the equation
\begin{equation}
\frac{\partial}{\partial t}u_{2k}(x, t) \, = \, (-1)^{k+1}  \frac{\partial^{2k}}{\partial x^{2k}}  u_{2k} (x, t), \qquad x \in \mathbb{R}.
\end{equation}
\end{te}
\begin{proof}
We start by evaluating the Fourier transform
\begin{align}
& \mathbb{E}e^{i\xi \sum_{j=0}^{N\l t \gamma^{-2k\beta} \r} U_j^{2k}(1)Q_j^{\gamma, 2k\beta}} \notag \\
 = \, & \mathbb{E} \left[ \mathbb{E}\l e^{i\xi U^{2k}(1)Q^{\gamma, 2k\beta}} \r^{N \l t\gamma^{-2k\beta} \r} \right] \notag \\
= \, &\exp \ll -\frac{\lambda t}{\gamma^{2k\beta}} \l 1-\mathbb{E}e^{i\xi U^{2k}(1)Q^{\gamma, 2k\beta}} \r \rr \notag \\
= \, & \exp \ll -\frac{\lambda t}{\gamma^{2k\beta}} \int_\gamma^\infty dy \l 1-e^{-|\xi |^{2k}y^{2k}} \r \rr \frac{2k\beta \gamma^{2k\beta}}{y^{2k\beta + 1}} \notag \\
= \, & \exp \ll -\frac{\lambda t}{\gamma^{2k\beta }} \left[ \l 1-e^{-|\xi |^{2k}\gamma^{2k}} \r +  \int_\gamma^{\infty} dy \, e^{-|\xi |^{2k}y^{2k}} y^{2k-1-2k\beta} \, 2k \gamma^{2k\beta}  \right]  \rr
\end{align}
By taking the limit we have that
\begin{align}
\lim_{\gamma \to 0} e^{i\xi \sum_{j=0}^{N\l t \gamma^{-2k\beta} \r} U_j^{2k}(1)Q_j^{\gamma, 2k\beta}} \, = \, & e^{-\lambda t |\xi |^{2k} 2k \int_0^\infty e^{-|\xi |^{2k}y^{2k}}y^{2k(1-\beta )-1} dy} \notag \\
= \, & e^{-\lambda t |\xi |^{2k\beta} \int_0^\infty e^{-w} w^{-\beta} dw} \notag \\
= \, & e^{-\lambda t |\xi |^{2k\beta}\Gamma (1-\beta)}
\end{align}
which coincides with
\begin{align}
\mathbb{E}e^{i\xi U^{2k} \l H^\beta (t) \r} \, = \, \int_0^\infty e^{-s\xi^{2k}} \Pr \ll H^\beta (t) \in ds \rr \, = \, e^{-t |\xi |^{2k\beta}}
\end{align}
since $\lambda = \frac{1}{\Gamma (1-\beta)}$.
\end{proof}
\begin{os}
For $\beta = \frac{1}{k}$ the composition $U^{2k} \l H^\beta (t) \r$  has Gaussian distribution. For $\beta = \frac{1}{2k}$ we have instead the Cauchy distribution and for $\beta = \frac{1}{4k}$ we extract the inverse Gaussian corresponding to the distribution of the first passage time of a Brownian motion. The case $\beta = \frac{1}{6k}$ yields the stable law with distribution
\begin{align}
f_{\frac{1}{3}} (x) \, = \, \frac{t}{x\sqrt[3]{3x}} \textrm{ Ai}\l \frac{t}{\sqrt[3]{3x}} \r
\end{align}
where Ai denotes the Airy function (see \cite{ecporsdov}).
\end{os}
In order to arrive at asymmetric higher-order fractional pseudoprocesses we construct pseudo random walks by adapting the Feller approach (used for asymmetric stable laws) to our context.
This means that we combine independent pseudo random walks with suitable trigonometric weights as in \eqref{14}.
\begin{te}
\label{teoremarwfellerdispari}
Let $X_j^{\gamma, (2k+1)\beta}$ and $Y_j^{\gamma, (2k+1)\beta}$ be i.i.d. r.v.'s with distribution
\begin{align}
\Pr \ll X^{\gamma, (2k+1)\beta} > w \rr \, = \, 
\begin{cases}
\l \frac{\gamma}{w} \r^{(2k+1)\beta}, \qquad & w > \gamma \\
1, & w < \gamma,
\end{cases}
\end{align}
and let $U^{2k+1}(t)$, $t>0$, be a pseudoprocess of odd-order $2k+1$, $k \in \mathbb{N}$. For $0 < \beta < 1$ and $-\beta < \theta < \beta$ we have that
\begin{align}
 \lim_{\gamma \to 0} & \left[  \l \frac{\sin \frac{\pi}{2}(\beta - \theta)}{\sin \pi \beta} \r^{\frac{1}{(2k+1)\beta}} \sum_{j=0}^{N \l t \gamma^{-(2k+1)\beta} \r} X_j^{\gamma, (2k+1)\beta} U_j^{2k+1} (1) \right. \notag \\
&  \left.  - \l \frac{\sin \frac{\pi}{2}(\beta + \theta)}{\sin \pi \beta} \r^{\frac{1}{(2k+1)\beta}} \sum_{j=0}^{N \l t \gamma^{-(2k+1)\beta} \r} Y_j^{\gamma, (2k+1)\beta} U_j^{2k+1} (1) \right] \stackrel{\textrm{law}}{=} \, & Z^{\beta (2 k+1), \theta}
\label{mammamiabella}
\end{align}
where
\begin{equation}
\mathbb{E}e^{i\xi Z^{\beta (2 k+1), \theta}} \, = \, e^{-t |\xi |^{(2k+1)\beta} e^{\frac{i\pi \theta}{2}}}
\end{equation}
\end{te}
\begin{proof}
The Fourier transform of \eqref{mammamiabella} is written as
\begin{align}
&\mathbb{E}e^{i\xi \l \frac{\sin \frac{\pi}{2}(\beta - \theta)}{\sin \pi \beta} \r^{\frac{1}{(2k+1)\beta}} \sum_{j=0}^{N \l t \gamma^{-(2k+1)\beta} \r} X_j^{\gamma, (2k+1)\beta} U_j^{2k+1} (1)} \notag \\
\times \, & \mathbb{E}e^{-i\xi \l \frac{\sin \frac{\pi}{2}(\beta + \theta)}{\sin \pi \beta} \r^{\frac{1}{(2k+1)\beta}} \sum_{j=0}^{N \l t \gamma^{-(2k+1)\beta} \r} Y_j^{\gamma, (2k+1)\beta} U_j^{2k+1} (1)}
\end{align}
where the first member is given by
\begin{align}
&\mathbb{E}e^{i\xi \l \frac{\sin \frac{\pi}{2}(\beta - \theta)}{\sin \pi \beta} \r^{\frac{1}{(2k+1)\beta}} \sum_{j=0}^{N \l t \gamma^{-(2k+1)\beta} \r} X_j^{\gamma, (2k+1)\beta} U_j^{2k+1} (1)} \, = \notag \\
= \, & \exp \ll -\frac{\lambda t}{\gamma^{(2k+1)\beta}} \l 1-\mathbb{E}e^{i\xi \l \frac{\sin \frac{\pi}{2}(\beta - \theta)}{\sin \pi \beta} \r^{\frac{1}{(2k+1)\beta}} U^{2k+1}(1) X^{(2k+1)\beta}} \r \rr \notag \\
= \, & \exp \ll - \frac{\lambda t}{\gamma^{(2k+1)\beta}} \int_\gamma^\infty \l 1-e^{i\xi^{2k+1} \l \frac{\sin \frac{\pi}{2}(\beta - \theta)}{\sin \pi \beta} \r^{\frac{1}{\beta}} y^{2k+1}} \r \frac{\gamma^{(2k+1)\beta}}{y^{(2k+1)\beta +1}} (2k+1)\beta  \rr \notag \\ 
\stackrel{\gamma \to 0}{\longrightarrow} \, & \exp \ll -\lambda t \, i^{-\beta} \, \xi^{(2k+1)\beta} \l \frac{\sin \frac{\pi}{2} (\beta - \theta)}{\sin \pi\beta} \r \int_0^\infty e^{-w} w^{-\beta }  dw\rr \notag \\
= \, & e^{-\lambda t |\xi |^{(2k+1)\beta} e^{-\frac{i\pi \beta \textrm{ sign}(\xi)}{2}}\l \frac{\sin \frac{\pi}{2} (\beta - \theta)}{\sin \pi\beta} \r \Gamma (1-\beta)} .
\end{align}
The second member of \eqref{mammamiabella} becomes, by performing a similar calculation,
\begin{align}
&\mathbb{E}e^{-i\xi \l \frac{\sin \frac{\pi}{2}(\beta + \theta)}{\sin \pi \beta} \r^{\frac{1}{(2k+1)\beta}} \sum_{j=0}^{N \l t \gamma^{-(2k+1)\beta} \r} Y_j^{\gamma, (2k+1)\beta} U_j^{2k+1} (1)} \notag \\
 \stackrel{\gamma \to 0}{\longrightarrow} \, & e^{-\lambda t |\xi |^{(2k+1)\beta} e^{\frac{i\pi \beta \textrm{ sign}(\xi)}{2}}\l \frac{\sin \frac{\pi}{2} (\beta + \theta)}{\sin \pi\beta} \r \Gamma (1-\beta)}.
\end{align}
and thus for $\lambda = \frac{1}{\Gamma (1-\beta)}$ we obtain that
\begin{align}
& \mathbb{E}e^{i\xi \l \frac{\sin \frac{\pi}{2}(\beta - \theta)}{\sin \pi \beta} \r^{\frac{1}{(2k+1)\beta}} \sum_{j=0}^{N \l t \gamma^{-(2k+1)\beta} \r} X_j^{\gamma, (2k+1)\beta} U_j^{2k+1} (1)} \notag \\
\times & \mathbb{E}e^{-i\xi \l \frac{\sin \frac{\pi}{2}(\beta + \theta)}{\sin \pi \beta} \r^{\frac{1}{(2k+1)\beta}} \sum_{j=0}^{N \l t \gamma^{-(2k+1)\beta} \r} Y_j^{\gamma, (2k+1)\beta} U_j^{2k+1} (1)} \notag \\
\stackrel{\gamma \to 0}{\longrightarrow} \, & e^{- t |\xi |^{(2k+1)\beta} e^{-\frac{i\pi \beta \textrm{ sign}(\xi)}{2}}\l \frac{\sin \frac{\pi}{2} (\beta - \theta)}{\sin \pi\beta} \r} e^{- t |\xi |^{(2k+1)\beta} e^{\frac{i\pi \beta \textrm{ sign}(\xi)}{2}}\l \frac{\sin \frac{\pi}{2} (\beta + \theta)}{\sin \pi\beta} \r} \notag \\
= \, & e^{-t|\xi |^{(2k+1)\beta} e^{\frac{i \pi \theta}{2} \textrm{ sign} (\xi)}}
\end{align}
\end{proof}
By considering symmetric pseudo random walks with the Feller construction we arrive in the next theorem at symmetric pseudoprocesses with time scale equal to $\frac{\cos \frac{\pi \beta}{2}}{\sin \frac{\pi \beta}{2}}$, $0<\beta<1$, $-\beta < \theta < \beta$.
\begin{te}
Let $X_j^{\gamma, 2\beta k}$ and $Y_j^{\gamma, 2\beta k}$ be i.i.d. r.v.'s with distribution
\begin{align}
\Pr \ll X^{\gamma, 2\beta k} > w \rr \, = \, 
\begin{cases}
\l \frac{\gamma}{w} \r^{2\beta k}, \qquad & w > \gamma \\
1, & w < \gamma,
\end{cases}
\end{align}
and let $U^{2k}(t)$, $t>0$, be a pseudoprocess of order $2k$, $k \in \mathbb{N}$. If $N(t)$ is a homogeneous Poisson process, with parameter $\lambda = \frac{1}{\Gamma (1-\beta)}$, independent from $X_j^{\gamma, 2\beta k}$ and $Y_j^{\gamma, 2\beta k}$ we have that
\begin{align}
\lim_{\gamma \to 0} & \left[ \l \frac{\sin \frac{\pi}{2}(\beta - \theta)}{\sin \pi \beta} \r^{\frac{1}{2\beta k}} \sum_{j=0}^{N \l t \gamma^{-2\beta k} \r} X_j^{\gamma, 2\beta k} \, U_j^{2k} (1) \, \right.  \notag \\
& \left. + \l \frac{\sin \frac{\pi}{2}(\beta + \theta)}{\sin \pi \beta} \r^{\frac{1}{2\beta k}} \sum_{j=0}^{N \l t \gamma^{-2\beta k} \r} Y_j^{\gamma, 2\beta k} \, U_j^{2k} (1) \right] \,  \stackrel{ \textrm{ law }}{=} \, Z^{2k\beta, \theta}, \qquad t >0,
\label{323}
\end{align}
for $0< \beta < 1$ and $-\beta < \theta < \beta$ and
\begin{align}
\mathbb{E}e^{i\xi Z^{2k\beta, \theta}} \, = \, e^{-t|\xi |^{2k\beta} \frac{\cos \frac{\pi}{2}\theta}{\cos \frac{\pi}{2}\beta}}
\end{align}
\end{te}
\begin{proof}
The Fourier transform of \eqref{323} is written as
\begin{align}
 \mathbb{E}e^{i\xi \l \frac{\sin \frac{\pi}{2}(\beta - \theta)}{\sin \pi \beta} \r^{\frac{1}{2\beta k}} \sum_{j=0}^{N \l t \gamma^{-2\beta k} \r} X_j \, U_j^{2k} (1)} \; \mathbb{E}e^{i\xi \l \frac{\sin \frac{\pi}{2}(\beta + \theta)}{\sin \pi \beta} \r^{\frac{1}{2\beta k}} \sum_{j=0}^{N \l t \gamma^{-2\beta k} \r} Y_j \, U_j^{2k} (1)}
 \end{align}
 where the first member is given by
 \begin{align}
& \mathbb{E}e^{i\xi \l \frac{\sin \frac{\pi}{2}(\beta - \theta)}{\sin \pi \beta} \r^{\frac{1}{2\beta k}} \sum_{j=0}^{N \l t \gamma^{-2\beta k} \r} X_j^{\gamma, 2\beta k} \, U_j^{2k} (1)} = \notag \\
= \, &  \exp \ll -\frac{\lambda t}{\gamma^{2k\beta}} \left[ 1-\mathbb{E}e^{i\xi \frac{\sin \frac{\pi}{2}(\beta - \theta)}{\sin \pi \beta} U^{2k}(1)X^{\gamma, 2\beta k}} \right] \rr \notag \\
= \, & \exp \ll -\frac{\lambda t}{\gamma^{2k\beta}} \left[ \int_\gamma^\infty e^{-\left| \xi \l \frac{\sin \frac{\pi}{2} (\beta - \theta)}{\sin \pi\beta} \r^{\frac{1}{2k\beta}} \right|^{2k} y^{2k}} (2k\beta) \frac{\gamma^{2k\beta}}{y^{2k\beta +1}} dy \right] \rr \notag \\
\stackrel{\gamma \to 0}{\longrightarrow} \, & \exp \ll -\lambda t |\xi |^{2k} \l \frac{\sin \frac{\pi}{2}(\beta - \theta)}{\sin \pi\beta} \r^{\frac{1}{\beta}} 2k \int_0^\infty e^{-| \xi |^{2k} \l \frac{\sin \frac{\pi}{2}(\beta - \theta)}{\sin \pi\beta} \r^{\frac{1}{\beta}}  y^{2k}} y^{2k-1-2k\beta} dy \rr \notag \\
= \, & e^{-\lambda t |\xi |^{2k\beta} \Gamma (1-\beta) \left[ \frac{\sin \frac{\pi}{2}(\beta - \theta)}{\sin \pi \beta}\right] } 
\end{align}
and by similar calculations the second member becomes
\begin{align}
\mathbb{E}e^{i\xi \l \frac{\sin \frac{\pi}{2}(\beta + \theta)}{\sin \pi \beta} \r^{\frac{1}{2\beta k}} \sum_{j=0}^{N \l t \gamma^{-2\beta k} \r} Y_j^{\gamma, 2\beta k} \, U_j^{2k} (1)} \, \stackrel{\gamma \to 0}{\longrightarrow}  \, e^{-\lambda t |\xi |^{2k\beta} \Gamma (1-\beta) \left[ \frac{\sin \frac{\pi}{2}(\beta + \theta)}{\sin \pi \beta}\right] } .
\end{align}
Thus we have that
\begin{align}
& \mathbb{E}e^{i\xi \l \frac{\sin \frac{\pi}{2}(\beta - \theta)}{\sin \pi \beta} \r^{\frac{1}{2\beta k}} \sum_{j=0}^{N \l t \gamma^{-2\beta k} \r} X_j^{\gamma, 2\beta k} \, U_j^{2k} (1)} \; \mathbb{E}e^{i\xi \l \frac{\sin \frac{\pi}{2}(\beta + \theta)}{\sin \pi \beta} \r^{\frac{1}{2\beta k}} \sum_{j=0}^{N \l t \gamma^{-2\beta k} \r} Y_j^{\gamma, 2\beta k} \, U_j^{2k} (1)} \notag \\
\stackrel{\gamma \to 0}{\longrightarrow} \, &  e^{-\lambda t |\xi |^{2k\beta} \Gamma (1-\beta) \left[ \frac{\sin \frac{\pi}{2}(\beta - \theta)}{\sin \pi \beta} + \frac{\sin \frac{\pi}{2}(\beta + \theta)}{\sin \pi \beta}\right] } \notag \\
 = \, & e^{-t |\xi |^{2\beta k} \frac{\cos \frac{\pi}{2}\theta}{\cos \frac{\pi}{2}\beta}}
\end{align}
\end{proof}

\section{Governing equations}
In the previous section we obtained fractional pseudoprocesses as limit of suitable pseudo random walks. In this section we will show that the limiting fractional pseudoprocesses obtained before have signed density satisfying space-fractional heat-type equations of higher-order with Riesz or Feller fractional derivatives. The order of fractionality of the governing equations is a positive real number and this is the major difference with respect to the pseudoprocesses considered so far in the literature.

We start by examining space fractional higher-order equations of order $2k\beta$, $\beta \in (0,1)$, $k \in \mathbb{N}$, which interpolate equations of the form \eqref{11}.
\begin{te}
The solution to the initial-value problem
\begin{align}
\begin{cases}
\frac{\partial}{\partial t} v_{2k}^\beta (x, t) \, = \, \frac{\partial^{2k\beta}}{\partial |x|^{2k\beta}} v_{2k}^\beta (x, t), \qquad x \in \mathbb{R}, t>0, k \in \mathbb{N}, \beta \in (0,1) \\
v_{2k}^\beta (x, 0) \, = \, \delta (x)
\end{cases}
\label{problemapari}
\end{align}
can be written as
\begin{align}
v_{2k}^\beta (x, t) \, = \, & \frac{1}{\pi x} \mathbb{E} \left[ \sin \l x G^{2k} \l \frac{1}{H^\beta (t)} \r \r \right] \notag \\
= \, & \frac{1}{\pi x} \mathbb{E} \left[ \sin \l x G^{2k\beta} \l \frac{1}{t} \r \r \right]
\label{rappresentpari}
\end{align}
and coincides with the law of the pseudoprocess
\begin{align}
V^{2k\beta} (t) \, = \, U^{2k} \l H^\beta (t) \r, \qquad t>0,
\end{align}
where $U^{2k}$ is related to equation \eqref{11} for $m=2k$ and $H^\beta$ is a stable subordinator independent from $U^{2k}$. $G^\gamma \l t \r$ is a gamma r.v. with density
\begin{equation}
g^\gamma (x, t) \, = \, \gamma \frac{x^{\gamma -1}}{t} e^{-\frac{x^\gamma}{t}}, \qquad x>0, t>0, \gamma >0.
\end{equation}
\end{te}
\begin{proof}
The Fourier transform of \eqref{problemapari} leads to the Cauchy problem
\begin{align}
\begin{cases}
\frac{\partial}{\partial t} \widehat{v}_{2k}^\beta (\xi, t) \, = \, - | \xi |^{2k\beta} \widehat{v}_{2k}^\beta (\xi, t) \\
\widehat{v}_{2k}^\beta (\xi, 0) \, = \, 1,
\end{cases}
\end{align}
whose unique solution reads
\begin{align}
\mathbb{E}e^{i\xi V^{2k\beta} (t)} \, = \, & \int_{\mathbb{R}} dx \, e^{i\xi x} \int_0^\infty ds \, u_{2k} (x, s) \, h_\beta (s, t) \notag \\
= \, & \int_0^\infty ds \, e^{-s\xi^{2k}} h_{\beta} (s, t) \, = \,  e^{-t|\xi|^{2k\beta}}.
\label{carattpari}
\end{align}
In \eqref{carattpari} $u_{2k}$ is the density of $U^{2k}$ and $h_\beta(x, t)$ is the probability density of the subordinator $H^\beta$.
Now we show that the Fourier transform of \eqref{rappresentpari} coincides with \eqref{carattpari}. We have that
\begin{align}
\widehat{v}_{2k}^\beta (\xi, t) \, = \, & \int_{\mathbb{R}} dx e^{i\xi x} \, \frac{1}{\pi x} \mathbb{E} \left[ \sin \l x G^{2k} \l \frac{1}{H^\beta (t)} \r \r \right] \notag \\
= \, & \int_{\mathbb{R}} dx \, e^{i\xi x} \left[ \int_0^\infty \int_0^\infty \frac{\sin xy}{\pi x} \Pr \ll G^{2k} \l \frac{1}{s} \r \in dy \rr \, \Pr \ll H^\beta (t) \in ds \rr \right] \notag \\
= \, & \int_0^\infty \int_0^\infty \Pr \ll G^{2k} \l \frac{1}{s} \r \in dy \rr \, \Pr \ll H^\beta (t) \in ds \rr \left[ \int_{\mathbb{R}} dx \, e^{i\xi x} \frac{\sin xy}{\pi x} \right].
\label{27}
\end{align}
By considering that the Heaviside function
\begin{align}
\mathcal{H}_\alpha (z) \, = \,
\begin{cases}
1, \qquad & z > \alpha, \\
0, & z < \alpha
\end{cases}
\end{align}
can be represented as
\begin{align}
\mathcal{H}_\alpha (z) \, = \, \frac{1}{2\pi} \int_{\mathbb{R}} dw \, e^{iwz} \frac{e^{-i\alpha w}}{iw} \, = \, - \frac{1}{2\pi} \int_{\mathbb{R}} dw \, e^{-iwz} \frac{e^{i\alpha w}}{iw},
\end{align}
we obtain that formula \eqref{27} becomes
\begin{align}
& \widehat{v}_{2k}^\beta (\xi, t) \, = \notag \\
= \, & \int_0^\infty \int_0^\infty \Pr \ll G^{2k} \l \frac{1}{s} \r \in dy \rr \, \Pr \ll H^\beta (t) \in ds \rr \left[ \mathcal{H}_{-y} (\xi) - \mathcal{H}_y (\xi) \right] \notag \\
= \, & \int_0^\infty \int_0^\infty \Pr \ll G^{2k} \l \frac{1}{s} \r \in dy \rr \, \Pr \ll H^\beta (t) \in ds \rr \left[ \mathbb{I}_{[-\xi, + \infty]} (y) - \mathbb{I}_{[-\infty, \xi]} (y)  \right] \notag \\
= \, & \int_0^\infty \int_0^\infty dy \, ds \l 2ksy^{2k-1} e^{-sy^{2k}} \r \mathbb{I}_{[0, \infty]} (y) \left[ \mathbb{I}_{[-\xi, + \infty]} (y) - \mathbb{I}_{[-\infty, \xi]} (y)  \right] h_\beta (s, t).
\label{mammamia}
\end{align}
For $\xi > 0$ \eqref{mammamia} becomes
\begin{align}
\widehat{v}_{2k}^\beta (\xi, t) \, = \,  & \int_0^\infty ds  \left[ 1 - \int_0^\xi dy \, 2ks y^{2k-1} e^{-y^{2k}s}  \right] h_\beta (s, t) \notag \\
= \, & \int_0^\infty ds \,  e^{-\xi^{2k}s}   h_\beta (s, t) \, = \,  e^{-t| \xi |^{2k\beta}},
\end{align}
and for $\xi < 0$ \eqref{mammamia} is
\begin{align}
\widehat{v}_{2k}^\beta (\xi, t) \, = \, & \int_0^\infty ds \, \left[ \int_{-\xi}^\infty 2ksy^{2k-1} e^{-y^{2k}s} \right] h_\beta (s, t) \notag \\
= \, & \int_0^\infty ds \, e^{-| \xi |^{2k} s } \, h_\beta (s, t) \, = \,  e^{-t |\xi |^{2k\beta}}.
\end{align}
Since
\begin{align}
\Pr \ll  G^{2k} \l \frac{1}{H^\beta (t)} \r \in dy \rr /dy \, = \, & 2ky^{2k-1} \int_0^\infty se^{-sy^{2k}} h_\beta (s, t) \, ds \notag \\
 = \, & -\frac{\partial}{\partial y} \int_0^\infty e^{-sy^{2k}} h_\beta (s, t) \, ds \notag \\
  = \, & -\frac{\partial}{\partial y} e^{-y^{2k\beta}t} \notag \\ 
  = \, &  \Pr \ll  G^{2k\beta} \l \frac{1}{t} \r \in dy \rr /dy
\end{align}
the second form of the solution \eqref{rappresentpari} follows immediately.
\end{proof}
For $k \geq 1$, $\beta \in \l 0, \frac{1}{k} \right]$ the solutions \eqref{rappresentpari} are densities of symmetric random variables, while for $\beta > \frac{1}{k}$ the functions \eqref{rappresentpari} are sign-varying. Clearly for $\beta = 1$ we obtain the solution of even-order heat-type equations discussed in \cite{ecporsdov}. As far as space-fractional higher-order heat-type equations we have the result of the next theorem where the governing fractional operator $\mathfrak{R}$ is obtained as a suitable combination of Weyl derivatives. The operator $\mathfrak{R}$ governing the fractional pseudoprocesses appearing in Theorem \ref{pseudowalkdispari} is explicitely written for $\ll p, q \in [0,1] : p + q = 1 \rr$, $\ll \beta \in (0,1), \, k \in \mathbb{N}: m-1 < \beta (2k+1) < m, \, m \in \mathbb{N} \rr$ as
\begin{align}
& \mathfrak{R} \, v_{2k+1}^\beta (x, t) \, = \notag \\
 = \,  & - \frac{1}{\cos \frac{\pi \beta}{2}} \left[ p \, e^{i \pi \beta k} \, \frac{^+\partial^{\beta(2k+1)}}{\partial x^{\beta(2k+1)}} + q \, e^{-i \pi \beta k} \,  \frac{^-\partial^{\beta(2k+1)}}{\partial x^{\beta(2k+1)}} \right]  v_{2k+1}^\beta (x, t)   \notag \\
 = \, & - \frac{1}{\cos \frac{\pi \beta}{2} \Gamma (m-(2k+1)\beta))} \frac{\partial^m}{\partial x^m}   \left[ e^{i \pi \beta k}  p  \int_{-\infty}^x \frac{v_{2k+1}^\beta (y, t)}{(x-y)^{(2k+1)\beta-m+1}} dy \right. \notag \\
 &  \left. + q \, e^{-i \pi \beta k} \, (-1)^m \int_x^\infty \frac{v_{2k+1}^\beta(y, t)}{(y-x)^{(2k+1)\beta -m+1}} dy \right],
 \label{sommaweylpesata}
\end{align}
where the left and right Weyl fractional derivatives appear.
\begin{te}
The solution to the problem
\begin{align}
\begin{cases}
\frac{\partial }{\partial t} v_{2k+1}^\beta (x, t) \, = \, \mathfrak{R} \; v_{2k+1}^\beta (x, t), \qquad x \in \mathbb{R}, t>0, \beta \in (0,1), k \in \mathbb{N}, \\
v_{2k+1}^\beta (x, 0) \, = \, \delta(x),
\end{cases}
\label{problemapq}
\end{align}
is given by the signed law of the pseudoprocess
\begin{align}
\bar{V}^{\beta (2k+1)} (t) \,
 = \, U_1^{2k+1} \l H_1^\beta \l \frac{p t}{\cos \frac{\beta \pi}{2}} \r \r - \, U_2^{2k+1} \l H_2^\beta \l \frac{qt}{\cos \frac{\beta \pi}{2}} \r \r,
 \label{vbarrato}
 \end{align}
where $U_1^{2k+1}$, $U_2^{2k+1}$ are independent odd-order pseudoprocesses and $H_1^\beta$, $H_2^\beta$, are independent stable subordinators.
\end{te}
\begin{proof}
The Fourier transform of \eqref{sommaweylpesata} is written as
\begin{align}
& \mathcal{F} \left[ \mathfrak{R} \, v_{2k+1}^\beta (x, t) \right] (\xi) \, = \notag \\
= \, & \mathcal{F} \left[ - \frac{1}{\cos \frac{\pi \beta}{2}} \left[ p \, e^{i \pi \beta k} \,  \frac{^+\partial^{\beta(2k+1)}}{\partial x^{\beta(2k+1)}} + q \, e^{-i \pi \beta k} \,  \frac{^-\partial^{\beta(2k+1)}}{\partial x^{\beta(2k+1)}} \right] v_{2k+1}^\beta (x, t) \right] (\xi) \notag \\
= \, & -\frac{1}{\cos \frac{\beta \pi}{2}} \left[ p \, e^{i \pi \beta k} \,  (-i\xi)^{\beta (2k+1)} + q \, e^{-i \pi \beta k} \,  (i\xi)^{\beta (2k+1)} \right] \widehat{v}_{2k+1}^\beta (\xi, t) \notag \\
= \, & - \frac{1}{\cos \frac{\beta \pi}{2}} | \xi |^{\beta (2k+1)} \left[ pe^{-\frac{i\pi \beta}{2} \textrm{ sign}(\xi)} + q e^{\frac{i\pi \beta}{2} \textrm{ sign}(\xi)}  \right] \widehat{v}_{2k+1}^\beta (\xi, t) \notag \\
= \, & -|\xi |^{\beta (2k+1)} \l 1- i \textrm{ sign}(\xi) \, (p-q) \tan \frac{\pi \beta}{2} \r \widehat{v}_{2k+1}^\beta (\xi, t)
\end{align}
and therefore we have that
\begin{equation}
\widehat{v}^\beta_{2k+1} (\xi, t)  \, = \, e^{-t| \xi |^{\beta(2k+1)} \l 1-i \l p-q \r \textrm{ sign}(\xi) \tan \frac{\beta \pi}{2} \r}
\end{equation}
In view of \eqref{caratteristicapseudosimm} we get
\begin{align}
\mathbb{E}e^{i\xi \bar{V}^{(2k+1)\beta}(t)} \, = \, & \mathbb{E}e^{i\xi V^{(2k+1)\beta}\l \frac{t}{\cos \frac{\beta \pi}{2}} \r} \notag \\
= \, & e^{-t |\xi |^{\beta(2k+1)} \l 1 -i \textrm{ sign}(\xi) \, (p-q) \tan  \frac{\pi \beta}{2} \r}
\end{align}
and this confirms that the solution to \eqref{problemapq} is given by the law of the pseudoprocess \eqref{vbarrato}.
\end{proof}
\begin{os}
Since $\frac{e^{\pm i \pi k \beta}}{\cos \frac{\beta \pi}{2}} = \frac{1}{\cos \beta (2k+1) \frac{\pi}{2}}$ (because $e^{i\pi k\beta} = \l e^{i\pi} \r^{k\beta} = \l e^{-i\pi} \r^{k\beta} = e^{-i\pi k\beta}$) the operator \eqref{sommaweylpesata} takes the form of the Riesz fractional derivative of order $\beta(2k+1)$ when $p=q=\frac{1}{2}$.
\end{os}

We now pass to the derivation of the governing equation of the fractional pseudoprocesses studied in Theorem \ref{teoremarwfellerdispari}. We first recall the definition of the Feller space-fractional derivative which is
\begin{align}
^FD^{\beta, \theta} u(x) \, = \, - \left[ \frac{\sin \frac{\pi}{2}(\beta - \theta)}{\sin (\pi \beta)} \frac{^+\partial^\beta}{\partial x^\beta}  + \frac{\sin \frac{\pi}{2}(\beta + \theta)}{\sin (\pi \beta)} \frac{^-\partial^\beta}{\partial x^\beta}  \right]u(x).
\end{align}
We recall that
\begin{align}
\mathcal{F} \left[  ^FD^{\beta, \theta} u(x) \right] (\xi) \, = \, -|\xi|^{\beta} e^{\frac{i \pi \theta}{2} \textrm{ sign}(\xi) } \widehat{u}(\xi), 
\label{fourierfeller01}
\end{align}
as can be shown by means of the following calculation
\begin{align}
\int_{\mathbb{R}} dx \, e^{i\xi x} \;  ^FD^{\beta, \theta} u(x) \, = \, & - \left[ \frac{\sin \frac{\pi}{2}(\beta -  \theta)}{\sin (\pi \beta)} (-i\xi)^\beta  + \frac{\sin \frac{\pi}{2}(\beta + \theta)}{\sin (\pi \beta)} (i\xi)^\beta  \right] \widehat{u}(\xi) \notag \\
= \, & -\frac{|\xi|^{\beta}}{2i \sin \pi \beta} \left[ \l e^{\frac{i\pi}{2}\beta} e^{-\frac{i\pi}{2}\theta} - e^{-\frac{i\pi}{2}\beta} e^{\frac{i\pi}{2}\theta} \r e^{-\frac{i\pi}{2}\beta \textrm{ sign}(\xi)} + \right. \notag \\
& + \left. \l e^{\frac{i\pi}{2}\beta} e^{\frac{i\pi}{2}\theta} - e^{-\frac{i\pi}{2}\beta} e^{-\frac{i\pi}{2}\theta} \r e^{\frac{i\pi}{2}\beta \textrm{ sign}(\xi)} \right] \widehat{u}(\xi)  \notag \\
= \, & \begin{cases} -\xi^{\beta} e^{\frac{i\pi \theta}{2} } \widehat{u}(\xi), \qquad &\xi > 0, \\
-(-\xi)^{\beta} e^{-\frac{i\pi \theta}{2} } \widehat{u}(\xi), \qquad &\xi < 0
\end{cases} \notag \\
= \, & -|\xi|^{\beta} e^{\frac{i \pi \theta}{2} \textrm{ sign}(\xi) } \widehat{u}(\xi)
\end{align}
where we used the results of Theorem \ref{teoremasimboloweyl}.
The explicit form of the Fourier transform of the solution to 
\begin{equation}
\frac{\partial}{\partial t} u(x, t) \, = \, ^FD^{\beta, \theta} u(x, t), \qquad u(x, 0) \, = \, \delta(x), \qquad x \in \mathbb{R}, t>0,
\end{equation}
is written as
\begin{equation}
\widehat{u}(\xi, t) \, = \, e^{-|\xi |^\beta t e^{\frac{i\pi \theta}{2} \textrm{ sign}(\xi)} }
\label{stable}
\end{equation}
and for $\beta \in (0,2]$, $4m-1 < \theta < 4m+1$, $m \in \mathbb{N}$, represents the characteristic function of a stable r.v.. The last condition on $\theta$ is due to the fact that
\begin{equation}
\left| \widehat{u}(\xi, t) \right| \leq 1 \textrm{ if and only if } \cos \frac{\theta \pi}{2} \in (0,1].
\label{condition}
\end{equation}
The condition \eqref{condition} must be assumed also for $\beta > 2$ where \eqref{stable} however fails to be the characteristic function of a genuine r.v..
For $\theta = \beta <1$ \eqref{stable} becomes totally negatively skewed. By interchanging $\sin (\beta - \theta) \frac{\pi}{2}$ with $\sin (\beta + \theta) \frac{\pi}{2}$ we obtain instead
\begin{align}
\widehat{u}(\xi, t) \, = \, e^{-|\xi |^{\beta}t e^{-\frac{i\pi}{2}\theta \textrm{ sign}(\xi)} }
\end{align}
which is totally positively skewed for $\theta = \beta<1$.

We are now ready to prove the following Theorem.
\begin{te}
Let $Z^{\beta (2k+1), \theta}(t)$, $t>0$, be the limiting fractional pseudoprocess studied in Theorem \ref{teoremarwfellerdispari}. The signed density of $Z^{\beta (2k+1), \theta}(t)$ is the solution to
\begin{align}
\begin{cases}
\frac{\partial}{\partial t} z^{\beta(2k+1), \theta}(x, t) \, = \, ^FD^{\beta(2k+1), \theta} z^{\beta(2k+1), \theta}(x, t) \\
z^{\beta(2k+1), \theta}(x, 0) \, = \, \delta(x)
\end{cases}
\label{cauchyfeller}
\end{align}
and coincide with the signed distribution of the composition for $\beta \in (0,1)$, $-\beta <\theta < \beta$,
\begin{align}
\mathfrak{Z}^{\beta(2k+1), \theta} (t) \, = \,  U_1^{2k+1} \l H_1^\beta \l \frac{\sin \frac{\pi}{2}(\beta + \theta)}{\sin \pi \beta}t \r \r - U_2^{2k+1} \l H_2^\beta \l \frac{\sin \frac{\pi}{2}(\beta - \theta)}{\sin \pi \beta} t\r  \r,
\label{impiccio}
\end{align}
where $H^\beta_{j}$, $j=1, 2$ are independent stable r.v.'s and the independent pseudoprocesses $U_j^{2k+1}$, $j=1, 2$, are related to the odd-order heat-type equation
\begin{align}
\frac{\partial}{\partial t} u_{2k+1} (x, t) \, = \, (-1)^k \frac{\partial^{2k+1}}{\partial x^{2k+1}} u_{2k+1} (x, t).
\end{align}
The positivity of the time scales in \eqref{impiccio} implies that $-\beta < \theta < \beta$.
\end{te}
\begin{proof}
By profiting from the result \eqref{fourierfeller01} we note that the Fourier transform of \eqref{cauchyfeller} is written as
\begin{align}
\begin{cases}
\frac{\partial}{\partial t} \widehat{z}^{\beta(2k+1), \theta} (\xi, t) \, = \, -|\xi |^{\beta(2k+1)} e^{\frac{i\pi}{2}\theta \textrm{ sign}(\xi)} \; \widehat{z}^{\beta(2k+1), \theta} (\xi, t)\\
\widehat{z}^{\beta(2k+1), \theta} (\xi,0) \, = \, 1.
\end{cases}
\end{align}
which is satisfied by the Fourier transform
\begin{equation}
\widehat{z}^{\beta(2k+1), \theta} (\xi, t) \, = \, e^{-t|\xi|^{\beta(2k+1)} e^{\frac{i\pi}{2}\theta \textrm{ sign}(\xi)}}.
\label{adessobasta}
\end{equation}
We now prove that the Fourier transform of \eqref{impiccio} coincides with \eqref{adessobasta}. In view of the independence of the r.v.'s and pseudo r.v.'s involved we write that
\begin{align}
& \mathbb{E}e^{i\xi \mathfrak{Z}^{\beta(2k+1), \theta} (t)} \, = \, \mathbb{E}e^{i\xi \left[ U_1^{2k+1} \l H_1^\beta \l \frac{\sin \frac{\pi}{2}(\beta + \theta)}{\sin \pi \beta}t \r \r - U_2^{2k+1} \l H_2^\beta \l \frac{\sin \frac{\pi}{2}(\beta - \theta)}{\sin \pi \beta} t\r  \r \right]} \notag \\
= \, & \left[ \int_{\mathbb{R}} dx \, e^{i\xi x}  \int_0^\infty ds \, u^1_{2k+1} (x, s) h^1_\beta \l s, \frac{\sin \frac{\pi}{2}(\beta + \theta)}{\sin \pi \beta}t  \r \right] \notag \\
& \times  \left[ \int_{\mathbb{R}}  dx \, e^{-i\xi x} \int_0^\infty ds \, u^2_{2k+1} (x, s) h^2_\beta \l s, \frac{\sin \frac{\pi}{2}(\beta - \theta)}{\sin \pi \beta}t  \r \right] \notag \\
 = \, & \left[ \int_0^\infty e^{-i\xi^{2k+1}s} h^1_\beta \l s, \frac{\sin \frac{\pi}{2}(\beta + \theta)}{\sin \pi \beta}t  \r ds \right]  \left[ \int_0^\infty e^{i\xi^{2k+1}s} h^2_\beta \l s, \frac{\sin \frac{\pi}{2}(\beta + \theta)}{\sin \pi \beta}t  \r ds \right] \notag \\
= \, & e^{-t \frac{\sin \frac{\pi}{2}(\beta + \theta)}{\sin \pi \beta } \l i\xi^{2k+1} \r^\beta}  e^{-t \frac{\sin \frac{\pi}{2}(\beta - \theta)}{\sin \pi\beta} \l - i\xi^{2k+1} \r^\beta} \notag \\
= \, & e^{- \frac{t|\xi |^{\beta(2k+1)}}{\sin \pi \beta} \left[  \sin \frac{\pi}{2}(\beta + \theta)  e^{\frac{i\pi}{2}\beta \textrm{ sign}(\xi)}   + \sin \frac{\pi}{2} (\beta - \theta) e^{-\frac{i\pi}{2}\beta \textrm{ sign}(\xi)} \right] } \notag \\
= \, &  e^{- \frac{t|\xi |^{\beta(2k+1)}}{2i\sin \pi \beta} \left[  \l e^{\frac{i\pi}{2}\beta} e^{\frac{i\pi}{2}\theta} - e^{-\frac{i\pi}{2}\beta} e^{-\frac{i\pi}{2}\theta} \r e^{\frac{i\pi}{2}\beta \textrm{ sign}(\xi)} + \l e^{\frac{i\pi}{2}\beta} e^{-\frac{i\pi}{2}\theta} - e^{-\frac{i\pi}{2}\beta} e^{\frac{i\pi}{2}\theta} \r e^{-\frac{i\pi}{2}\beta \textrm{ sign}(\xi)}  \right] } \notag \\
= \, & e^{-t|\xi|^{\beta(2k+1)} e^{\frac{i\pi \theta}{2}\textrm{ sign}(\xi)} }
\end{align}
which coincides with \eqref{adessobasta}.
\end{proof}

\section{Some remarks}
We give various forms for the density $v^\gamma (x, t)$ of symmetric pseudoprocesses of arbitrary order $\gamma > 0$. For integer values of $\gamma = 2n$ or $\gamma = 2n+1$ the analysis of the structure of these densities is presented in \citet{ecporsdov}. We give here an analytical representation of $v^\gamma (x, t)$ for non-integer values of $\gamma$, which is an alternative to \eqref{rappresentpari}, as a power series and in integral form (involving the Mittag-Leffler functions). Furthermore, in Figure \ref{figura1} we give some curves for special values of $\gamma$. We also give the distribution of the sojourn time of compositions of pseudoprocesses with stable subordinators (totally positively skewed stable r.v.'s).
\begin{prop}
For $\gamma > 1$ the inverse of the Fourier transform
\begin{equation}
\widehat{v}^\gamma (\xi, t) \, = \, e^{-t|\xi |^\gamma}
\end{equation}
can also be written as
\begin{align}
v^\gamma (x, t) \, = \, & \frac{1}{\pi} \int_0^\infty \cos (\xi x) \, e^{-t \xi^\gamma} \, d\xi \notag \\
= \, & \frac{1}{\pi \gamma} \sum_{k=0}^\infty \frac{(-1)^k x^{2k}}{(2k)!} \frac{\Gamma \l \frac{2k+1}{\gamma} \r}{t^{\frac{2k+1}{\gamma}}} \notag \\
= \, & \frac{1}{\pi \gamma} \sum_{k=0}^\infty \frac{(-1)^k x^{2k}}{t^{\frac{2k+1}{\gamma}}} \frac{B \l \frac{2k+1}{\gamma}, \, (2k+1) \l 1-\frac{1}{\gamma} \r \r}{\Gamma \l (2k+1) \l 1-\frac{1}{\gamma} \r \r} \notag \\
= \, & \frac{1}{\pi \gamma}\sum_{k=0}^\infty \frac{(-1)^k x^{2k}}{t^{\frac{2k+1}{\gamma}}} \frac{1}{\Gamma \l (2k+1) \l 1-\frac{1}{\gamma} \r \r} \int_0^1 dy \, y^{\frac{2k+1}{\gamma}-1} (1-y)^{(2k+1) \l 1-\frac{1}{\gamma} \r -1} \notag \\
= \, & \frac{1}{\pi \gamma} \int_0^1 dy \, \sum_{k=0}^\infty \frac{(-1)^k \l x y^{\frac{1}{\gamma}} \l 1-y \r^{1-\frac{1}{\gamma}}  \r^{2k}}{t^{\frac{2k+1}{\gamma}} \Gamma \l (2k+1) \l 1-\frac{1}{\gamma} \r \r} \, y^{\frac{1}{\gamma}-1} \l 1-y \r^{-\frac{1}{\gamma}} \notag \\
= \, & \frac{t^{-\frac{1}{\gamma}}}{\pi \gamma} \int_0^1 dy \, E_{2\l 1-\frac{1}{\gamma} \r, 1-\frac{1}{\gamma}} \l -\l x y^{\frac{1}{\gamma}} (1-y)^{1-\frac{1}{\gamma}} \r^2 t^{-\frac{1}{\gamma}} \r \; y^{\frac{1}{\gamma}-1} \l 1-y \r^{-\frac{1}{\gamma}} \notag \\
\stackrel{w=y/(1-y)}{=} \, &  \frac{t^{-\frac{1}{\gamma}}}{\pi \gamma} \int_0^\infty dw \, E_{2\l 1-\frac{1}{\gamma} \r, 1-\frac{1}{\gamma}} \l -x^2 \l \frac{w^{\frac{1}{\gamma}}}{1+w} \r^2 t^{-\frac{1}{\gamma}} \r \frac{w^{\frac{1}{\gamma}}}{1+w} \frac{1}{w} 
\label{alternativarappr}
\end{align}
and for $\gamma < 2$ coincides with the characteristic function of symmetric stable processes.
\end{prop}
Formula \eqref{alternativarappr} is an alternative to the probabilistic representation \eqref{rappresentpari} for $\gamma = 2k\beta$. For $1<\gamma < 2$ it represents the density of a symmetric stable r.v..
\begin{os}
We note that
\begin{align}
v^\gamma (0, t) \, = \,  \frac{t^{-\frac{1}{\gamma}}}{\pi} \Gamma \l 1+ \frac{1}{\gamma} \r
\end{align}
as can be inferred from \eqref{alternativarappr}. In the neighbourhood of $x=0$ the density $v^\gamma (x, t)$ can be written as
\begin{align}
v^\gamma (x, t) \, \approx \, & \frac{1}{\pi \gamma} \l \frac{1}{t^{\frac{1}{\gamma}}} \Gamma \l \frac{1}{\gamma} \r - \frac{x^2}{2} \frac{\Gamma \l \frac{3}{\gamma} \r}{t^{\frac{3}{\gamma}}} \r \notag \\
= \, & v^\gamma (0, t) \, \l 1-x^2 \frac{C_\gamma}{2t^{\frac{2}{\gamma}}} \r
\end{align}
where
\begin{equation}
C_\gamma \, = \, \frac{\Gamma \l \frac{1}{\gamma} + \frac{1}{3} \r \Gamma \l \frac{1}{\gamma} + \frac{2}{3} \r \, 3^{\frac{3}{\gamma}-\frac{1}{2}}}{2\pi}.
\end{equation}
In the above calculation the triplication formula of the Gamma function (see \citet{lebedev} page 14) has been applied
\begin{equation}
\Gamma (z) \Gamma \l z+\frac{1}{3} \r \Gamma \l z+\frac{2}{3} \r \, = \, \frac{2\pi}{3^{3z-\frac{1}{2}}} \Gamma (3z).
\end{equation}
\end{os}

\begin{os}
For even-order pseudoprocesses $U^{2k}(t)$, $t>0$, the distribution of the sojourn time
\begin{equation}
\Gamma_t \l U^{2k} \r \, = \, \int_0^t \mathbb{I}_{[0, \infty)} \l U^{2k}(s) \r ds
\end{equation}
follows the arcsine law for all $n\geq 1$ (see \citet{krylov}). Therefore the distribution of the sojourn time of $U^{2k} \l H^\beta (t) \r$, $t>0$, $\beta \in (0,1)$, reads
\begin{align}
\Pr \ll \Gamma_t \l U^{2k} \l H^\beta \r \r \in dx \rr \, = \, &  \int_0^\infty \Pr \ll \Gamma_s \l U^{2k} \r \in dx \rr \, \Pr \ll H^\beta(t) \in ds \rr \notag \\
= \, & \frac{dx}{\pi} \int_x^\infty \frac{1}{\sqrt{x (s-x)}} \Pr \ll H^\beta (t) \in ds \rr.
\label{55}
\end{align}
In the odd-order case the distribution of the sojourn time
\begin{align}
\Gamma_t \l U^{2k+1} \r \, = \, \int_0^t \mathbb{I}_{[0, \infty)} \l U^{2k+1} \l s \r \r ds
\end{align}
is written as (see \citet{lachal2003})
\begin{equation}
\Pr \ll \Gamma_t \l U^{2k+1} \r \in dx \rr \, = \, dx \, \frac{\sin \frac{\pi}{2k+1}}{\pi} x^{-\frac{1}{2k+1}} \l t-x \r^{-\frac{2k}{2k+1}} \mathbb{I}_{(0,t)} (x)
\end{equation}
and thus we get
\begin{align}
\Pr \ll \Gamma_t \l U^{2k+1} \l H^\beta \r \r \in dx \rr \, = \, & \frac{dx \, \sin \frac{\pi}{2k+1}}{\pi} \int_x^\infty \frac{1}{\sqrt[2k+1]{x (s-x)^{2k}}} \, \Pr \ll H^\beta (t) \in ds \rr.
\end{align}
\end{os}
For $\beta = \frac{1}{2}$ the integral \eqref{55} can be evaluated explicitly
\begin{align}
\Pr \ll \Gamma_t \l U^{2k} \l H^{\frac{1}{2}} \r \r \in dx \rr \, = \, & \frac{dx}{\pi} \int_x^\infty \frac{1}{\sqrt{x(s-x)}} \frac{te^{-\frac{t^2}{2s}}}{\sqrt{2\pi s^3}} ds \notag \\
= \, & \frac{dx \, t}{\pi \sqrt{2\pi x}} \int_0^{\frac{1}{x}} \frac{e^{-\frac{t^2}{2}y}}{\sqrt{1-xy}} dy \notag \\
= \, &\frac{dx \, t}{\pi \sqrt{2\pi x^3}} \int_0^1 \frac{e^{-\frac{t^2}{2x}w}}{\sqrt{1-w}} dw \notag \\
= \, & \frac{ dx \, t}{\pi \sqrt{2\pi x^3}} \sum_{k=0}^\infty \l -\frac{t^2}{2x} \r^k \frac{1}{k!} \int_0^1 w^k \l 1-w \r^{-\frac{1}{2}} \, dw \notag \\
= \, & \frac{dx \, t}{\pi \sqrt{2x^3}} E_{1, \frac{3}{2}} \l -\frac{t^2}{2x} \r, \qquad x>0, t>0,
\end{align}
where
\begin{equation}
E_{\nu, \mu} (x) \, = \, \sum_{j=0}^\infty \frac{x^j}{\Gamma (j\nu+\mu)}, \qquad \nu, \mu > 0,
\end{equation}
is the Mittag-Leffler function.

\begin{figure} [htp!]
	\centering
		\caption{The density $v^\gamma (x, t)$ for $\gamma > 2$ displays an oscillating behaviour similar to that of the fundamental solution of even-order heat equations.}
		\includegraphics[scale=0.24]{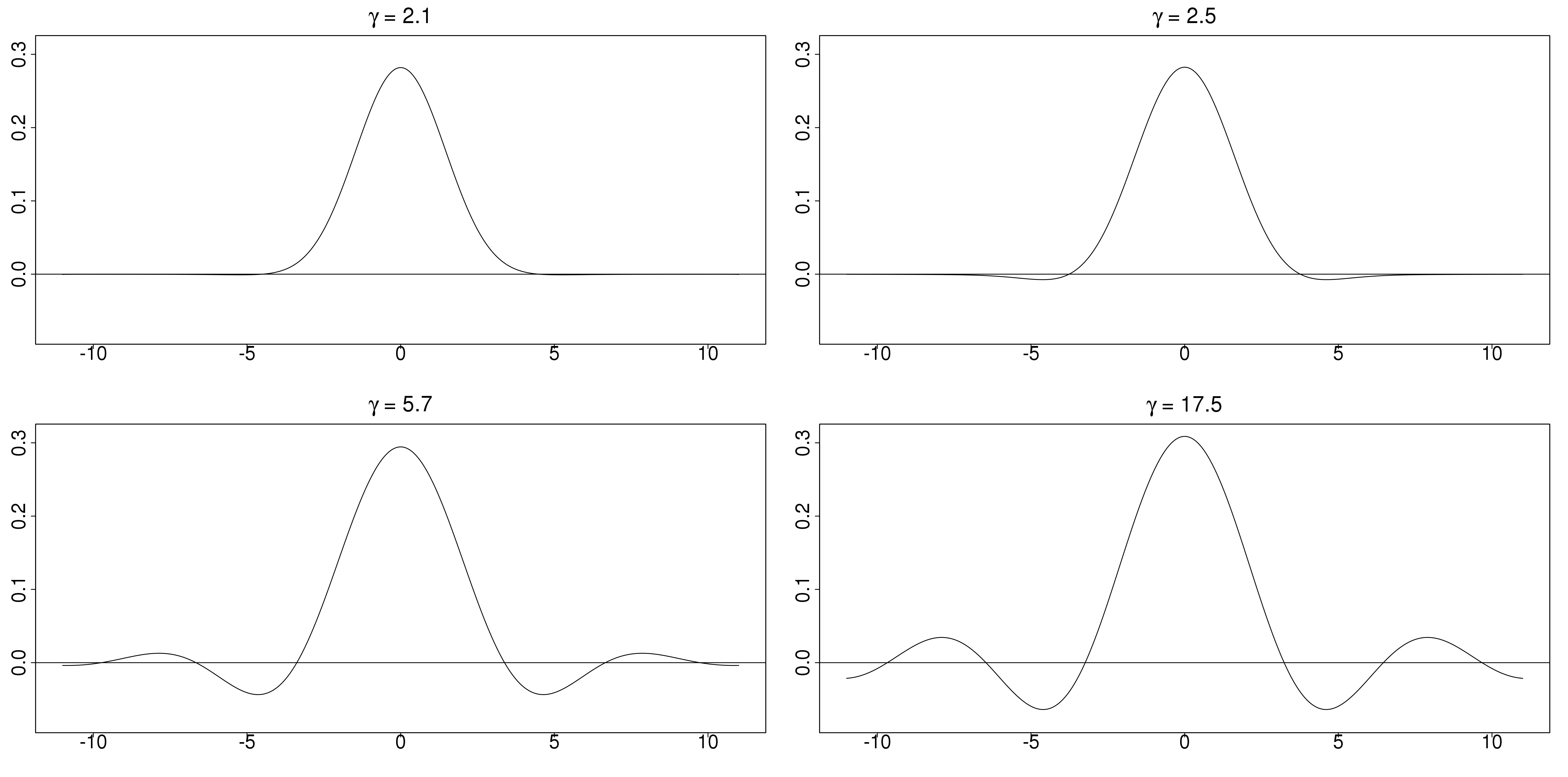} 
		\label{figura1}
\end{figure}

\section*{Ackwnoledgements}
The authors have benefited from fruitful discussions on the topics of this paper with Dr. Mirko D'Ovidio.

\end{document}